\documentclass[12pt]{article}

\usepackage{latexsym,amssymb,upref,amsmath,amsthm, amsfonts,authblk}
\usepackage{amssymb,amsmath,amsthm, calc, graphicx}
\usepackage{epsfig}
\usepackage{breqn}
\usepackage{footnpag}
\usepackage{rotating}
\usepackage{xcolor}
\usepackage{amsfonts}
\usepackage{setspace}
\usepackage{fullpage}
\usepackage{enumitem}
\usepackage{bbold}
\usepackage{comment}
\usepackage{pgf,tikz}
\usepackage[T1]{fontenc}
\usepackage{mathrsfs}
\usetikzlibrary{arrows}
\usepackage{yhmath}
\usepackage{hyperref}
\usepackage{authblk}
\usepackage{amsmath}
\usepackage{stackengine}

\setstackEOL{\\}
\usepackage{caption}
\usepackage{floatrow}
\usepackage{algorithm}
\usepackage{algpseudocode}
\usepackage{xpatch}
\algrenewcommand\alglinenumber[1]{{\sffamily\footnotesize#1}}
\makeatletter
\xpatchcmd{\algorithmic}{\itemsep\z@}{\itemsep=1ex plus2pt}{}{}
\makeatother

\algnewcommand\algorithmicinput{\textbf{Input:}}
\algnewcommand\algorithmicoutput{\textbf{Output:}}
\algnewcommand\Input{\item[\algorithmicinput]}%
\algnewcommand\Output{\item[\algorithmicoutput]}%

\newtheorem*{thm*}{Theorem*}
\newtheorem*{prop*}{Proposition}

\newtheorem{thm}{Theorem}
\newtheorem{prop}{Proposition}
\newtheorem{claim}{Claim}
\newtheorem{lemma}{Lemma}

\newtheorem{question}{Question}

\theoremstyle{definition}
\newtheorem{remark}{Remark}
\newtheorem{observation}{Observation}
\newtheorem{definition}{Definition}

\newcommand{\thistheoremname}{}
\newtheorem*{genericthm*}{\thistheoremname}
\newenvironment{namedthm*}[1]
{\renewcommand{\thistheoremname}{#1}%
	\begin{genericthm*}}
	{\end{genericthm*}}

\DeclareMathOperator{\ex}{ex}
\DeclareMathOperator{\codeg}{codeg}

\newcommand{\h}{\mathcal{H}}

\newcommand{\G}{\mathcal{G}}

\newcommand{\abs}[1]{\left\lvert{#1}\right\rvert}

\title{New bounds for a hypergraph Bipartite Tur\'an problem}

\author
{
Beka Ergemlidze
\thanks{ Alfr\'ed R\'enyi Institute of Mathematics, Budapest.
		E-mail: \texttt{beka.ergemlidze@gmail.com}} \qquad
Tao Jiang \thanks{Department of Mathematics, Miami University, Oxford, OH 45056, USA. E-mail: \texttt{jiangt@miamioh.edu}} \qquad
Abhishek Methuku \thanks{Department of Mathematics, \'Ecole Polytechnique F\'ed\'erale de Lausanne, Switzerland. E-mail: \texttt{abhishekmethuku@gmail.com}} 
}

\begin{document}

\maketitle

\begin{abstract}
Let $t$ be an integer such that $t\geq 2$. Let $K_{2,t}^{(3)}$ denote the triple system consisting of 
the $2t$ triples $\{a,x_i,y_i\}$, $\{b,x_i,y_i\}$ for $ 1 \le i \le t$, where the elements $a, b, x_1, x_2, \ldots, x_t,$ $y_1, y_2, \ldots, y_t$ are all distinct. Let $\ex(n,K_{2,t}^{(3)})$ denote
the maximum size of a triple system on $n$ elements that does not contain $K_{2,t}^{(3)}$.
This function was studied by Mubayi and Verstra\"ete \cite{MV}, where the special case $t=2$ was a problem of Erd\H{o}s \cite{Erdos} that was studied by various authors \cite{Furedi, MV, PV}.

Mubayi and Verstra\"ete proved that $\ex(n,K_{2,t}^{(3)})<t^4\binom{n}{2}$ and that
for infinitely many $n$, $\ex(n,K_{2,t}^{(3)})\geq \frac{2t-1}{3} \binom{n}{2}$. These bounds together with a standard argument show that
$g(t):=\lim_{n\to \infty} \ex(n,K_{2,t}^{(3)})/\binom{n}{2}$ exists and that 
\[\frac{2t-1}{3}\leq g(t)\leq t^4.\]
Addressing the question of Mubayi and Verstra\"ete on the growth rate of $g(t)$, we prove that 
as $t \to \infty$,
\[g(t) = \Theta(t^{1+o(1)}).\]
\end{abstract}

\section{Introduction}

An {\it $r$-graph} is an $r$-uniform hypergraph. Let $\mathcal F$ be a family of $r$-graphs and let $\ex(n, \mathcal F)$ denote the maximum number of edges in an $r$-graph on $n$ vertices containing no member of $\mathcal F$. We call $\ex(n,\mathcal F)$ the {\it Tur\'an number} of $\mathcal F$.
Determining the asymptotic order of $\ex(n,\mathcal F)$ is generally very
difficult. For an excellent survey on the study of hypergraph Tur\'an numbers, see \cite{Keevash}.
In this paper, we study a hypergraph Tur\'an problem that is motivated by the study of Tur\'an numbers of complete bipartite graphs as well as by a question of Erd\H{o}s.

\vspace{2mm}

\begin{definition}
Let $r\geq 3$ be an integer. 
Let $G$ be a bipartite graph with an ordered bipartition $(X,Y)$.
Suppose that $Y=\{y_1,\dots, y_m\}$. 
Let $Y_1,\dots, Y_m$ be disjoint sets of size $r-2$ that are disjoint from $X\cup Y$.
Let $G_{X,Y}^{(r)}$ denote the $r$-graph with vertex set $(X\cup Y)\cup (\bigcup_{i=1}^m Y_i)$
and edge set $\bigcup_{i=1}^m\{ e\cup Y_i: e\in E(G), y_i\in e\}$. 

Let $s,t\geq 2$ be positive integers. If $G$ is the complete bipartite graph with an ordered
bipartition $(X,Y)$ where $|X|=s, |Y|=t$, then let $G_{X,Y}^{(r)}$ be denoted by $K_{s,t}^{(r)}$.
\end{definition}

\begin{definition}
For all $n\geq r\geq 3$, let $f_r(n)$ denote the maximum number of edges in an $n$-vertex
$r$-graph containing no four edges $A,B,C,D$ with $A\cup B=C\cup D$ and 
$A\cap B=C\cap D=\emptyset$.
\end{definition}

Note that $f_3(n)=\ex(n,K_{2,2}^{(3)})$, and in general $f_r(n)\leq \ex(n,K_{2,2}^{(r)})$.
Erd\H{o}s \cite{Erdos} asked whether $f_r(n)=O(n^{r-1})$ when $r\geq 3$. F\"uredi \cite{Furedi}
answered Erd\H{o}s' question affirmatively.
More precisely, he showed that for integers $n,r$ with $r\geq 3$ and $n\geq 2r$,
 \begin{equation}\label{furedi-bounds}
 \binom{n-1}{r-1}+ \left \lfloor \frac{n-1}{r} \right \rfloor \le  f_r(n) < 3.5 \binom{n}{r-1}.
 \end{equation}

The lower bound is obtained by taking 
the family of all $r$-element subsets of $[n] := \{1,2, \ldots, n\}$ containing a
fixed element, say $1$, and adding to the family any collection of $\left \lfloor \frac{n-1}{r} \right \rfloor$ pairwise disjoint $r$-element subsets not containing $1$. For $r=3$, F\"uredi also gave an alternative
lower bound construction using Steiner systems. An {\it $(n,r,t)$-Steiner system} $S(n,r,t)$ is an $r$-uniform hypergraph on $[n]$ in which every $t$-element subset of $[n]$ is contained in exactly one hyperedge.
F\"uredi observed that if we replace every
hyperedge in  $S(n,5,2)$ by all its $3$-element subsets then the resulting triple system 
has $\binom{n}{2}$ triples and contains no copy of $K_{2,2}^{(3)}$. This slightly improves the lower bound in \eqref{furedi-bounds} for $r=3$ to $\binom{n}{2}$, for those $n$ for which $S(n,5,2)$ exists.
The upper bound in \eqref{furedi-bounds} was improved by Mubayi and Verstra\"ete \cite{MV} to 
$3 \binom{n}{r-1} + O(n^{r-2})$. They obtain this bound by first showing $f_3(n) =\ex(n,K_{2,2}^{(3)}) <3\binom{n}{2}+6n$, and then combining it with a simple reduction lemma. This was later improved to  $f_3(n)\leq \frac{13}{9} \binom{n}{2}$ by Pikhurko and Verstra\"ete \cite{PV}.

Motivated by F\"uredi's work, 
Mubayi and Verstra\"ete \cite{MV} initiated the study of the general problem of determining $\ex(n, K_{2,t}^{(r)})$ for any $t\geq 2$. They showed that for any $t \ge 2$ and $n\geq 2t$
$$\ex(n, K_{2,t}^{(3)}) < t^4 \binom{n}{2},$$
and that for infinitely many $n$, $\ex(n, K_{2,t}^{(3)}) \ge \frac{2t-1}{3} \binom{n}{2}$,
where the lower bound is obtained by replacing each hyperedge in  $S(n,2t+1,2)$ 
with all its $3$-element subsets.

Mubayi and Verstra\"ete noted that $g(t):=\lim_{n\to\infty} \ex(n, K_{2,t}^{(3)})/\binom{n}{2}$ exists and raised the question of determining the growth rate of $g(t)$. It follows from their results that 
\begin{equation} \label{MV-bounds}
 \frac{2t-1}{3} \le g(t) \le t^4.
 \end{equation}

In this paper, we prove that as $t \to \infty$,
\begin{equation} \label{our-bounds}
g(t) = \Theta(t^{1+o(1)}),
\end{equation}
showing that their lower bound is close to the truth. More precisely, we prove the following.

\begin{thm}
\label{mainthm}
For any $t \ge 2$, we have
$$\ex(n, K_{2,t}^{(3)}) \le \left( 15 t \log t + 40 t \right) n^2.$$
\end{thm}

\vspace{2mm}

\textbf{Notation.} Given a hypergraph (or a graph) $H$, throughout the paper, we also denote the set of its edges by $H$. For example $\abs{H}$ denotes the number of edges of $H$. Given two vertices $x,y$ in a graph $H$, let $N_H(x,y)$ denote the common neighborhood of $x$ and $y$ in $H$. We drop the subscript $H$ when the context is clear.

\vspace{2mm}

\section{Proof of Theorem \ref{mainthm}: $K_{2,t}^{(3)}$-free hypergraphs}
\label{mainthmsection}

We will use the a special case of a well-known result of Erd\H{o}s and Kleitman \cite{EK}.

\begin{lemma} 
\label{ErdosKleitman}
Let $H$ be a $3$-graph on $3n$ vertices. Then $H$ contains a $3$-partite $3$-graph, with all parts of size $n$, and with at least $\frac{2}{9} \abs{H}$ hyperedges.
\end{lemma}

Let us define the sets $A = \{a_1, a_2, \ldots, a_n \}$, $B = \{b_1, b_2, \ldots, b_n \}$ and $C = \{c_1, c_2, \ldots, c_n \}$. Throughout the proof we define various $3$-partite $3$-graphs whose parts are $A, B$ and $C$.

\vspace{2mm}

Suppose $H$ is a $K_{2,t}^{(3)}$-free $3$-partite $3$-graph on $3n$ vertices with parts $A, B$ and $C$. First let us show that it suffices to prove the following inequality. 
\begin{equation}
\label{eq:reductionto3-partite}
\abs{H} \le (30 t \log t + 80 t) n^2.
\end{equation}
It is easy to see that inequuity \eqref{eq:reductionto3-partite} and Lemma \ref{ErdosKleitman} together imply that any $K_{2,t}^{(3)}$-free $3$-graph on $3n$ vertices contains at most $\frac{9}{2} (30 t \log t + 80 t) n^2$ hyperedges, from which Theorem \ref{mainthm} would follow after replacing $3n$ by $n$. 

In the remainder of the section, we will prove \eqref{eq:reductionto3-partite}. Let us introduce the following notion of sparsity.

\begin{definition}[$q$-sparse and $q$-dense pairs]
Let $q$ be a positive integer. Let $G$ be a bipartite graph with parts $X, Y$. Let $x,y$ be two different vertices such that $x,y\in X$ or $x,y\in Y$. 
Then we call $\{x,y\}$ a \emph{$q$-dense} pair of $G$ if $\abs{N(x,y)} \ge q$. We call $\{x,y\}$ a $q$-\emph{sparse} pair of $G$ if $\abs{N(x,y)} < q$ but $x,y$ are still contained in a copy of $K_{2,q}$ in $G$. Note that it is possible that $\{x,y\}$ is neither $q$-sparse nor $q$-dense.
\end{definition}

The following Procedure $\mathcal P(q)$ about making a bipartite graph $K_{2,q}$-free lies at the heart of the proof. (We think of $q$ as the parameter of the Procedure $\mathcal P(q)$, that is changed throughout the proof.)
 
			\begin{algorithm}[H]
			\caption*{\textbf{Procedure $\mathcal P(q)$:} Making a graph $K_{2,q}$-free}
			\begin{algorithmic}
\smallskip			
	\Input{A bipartite graph $G$ with parts $A$ and $B$.}
	\State $\mathcal G \leftarrow G$, $\psi \leftarrow 1$.
	
	\State $F(x,y) \leftarrow \emptyset$ , $D(x,y) \leftarrow \emptyset$ and $S(x,y) \leftarrow \emptyset$ for every $x,y \in A$ and $x, y \in B$. 
	

	
	\While{$\psi = 1$}
	   \State $\psi \leftarrow 0$.
       \State \emph{\textbf{Step 1:}}
       
	  \State \Longunderstack[l]{For each $q$-sparse pair $\{x,y\}$ of $
	  \mathcal G$ such that $F(x,y) = \emptyset$, let $S(x,y)$ be the set of \\vertices spanned  by the $q$-dense pairs of $\mathcal G$ that are contained in $N_{\mathcal G}(x,y)$. \\ Let $F(x,y) \leftarrow \{ab \in \mathcal G \mid a \in \{x, y\} \text{ and } b \in S(x,y) \},$ and let $D(x,y)$ be a spanning \\forest of the graph formed by the dense pairs of $\mathcal G$ that are  contained in $S(x,y)$.}
      \State {\Longunderstack[l]{\textbf{If} there exists an edge $ab \in \mathcal G$ such that $ab$ is contained in $F(x,y)$ for at least $q/2$ \\different pairs $\{x,y\}$, where $x,y\in A$ or $x,y\in B$,}}
      \State \textbf{then} $\mathcal G \leftarrow \mathcal G \setminus \{ab\}$ and $\psi \leftarrow 1$.
      
      \State \emph{\textbf{Step 2:}}
      \State  \Longunderstack[l]{\textbf{If} there exists a set $M$ of edges in $\mathcal G$ such that removing all of the edges of $M$ from $\mathcal G$ \\ decreases the number of $q$-dense pairs by at least $\abs{M}/2$,}
      \State \textbf{then} $\mathcal G \leftarrow \mathcal G \setminus M$ and $\psi \leftarrow 1$.
    \EndWhile
  \State  $G' \leftarrow \mathcal G$
  \State $F'(x,y) \leftarrow F(x,y)$ for every  $x,y \in A$ and $x, y \in B$.
  \State $D'(x,y) \leftarrow D(x,y)$ for every  $x,y \in A$ and $x, y \in B$.
  \State $S'(x,y) \leftarrow S(x,y)$ for every  $x,y \in A$ and $x, y \in B$.
 \Output{The graph $G'$ and the sets $F'(x,y), D'(x,y), S'(x,y)$ for all $x,y \in A$ and $x, y \in B$.}
  \end{algorithmic}
\end{algorithm}

In the procedure $\mathcal P(q)$, initially for all the pairs $\{x,y\}$ (with $x,y \in A$ and $x,y \in B$) the sets $F(x,y), D(x,y), S(x,y)$ are set to be empty. Then as the edges are being deleted during the procedure, possibly, new $q$-sparse pairs $\{x,y\}$ are being created. When this happens, Step 1 redefines the sets $S(x,y), F(x,y), D(x,y)$ and gives them some non-empty values. (They get non-empty values due to the fact that $\{x,y\}$ is $q$-sparse, which implies that $\{x,y\}$ is contained in a copy of $K_{2,q}$, so there is at least one $q$-dense pair in the common neighborhood of $x, y$.) Therefore, these values stay unchanged throughout the rest of the procedure.
    
Notice that at the point $S(x,y)$ was redefined, the pair $\{x,y\}$ was $q$-sparse, so number of common neighbors is less than $q$. Therefore, as $S(x,y)$ is a subset of the common neighborhood of $x$ and $y$, we also have $\abs{S(x,y)} < q$. Moreover, since $D(x,y)$ is defined as a spanning forest with the vertex set $S(x,y)$, we have $\abs{D(x,y)} \le \abs{S(x,y)}$. Also, it easily follows from the definition of $F(x,y)$ that  $\abs{F(x,y)} = 2 \abs{S(x,y)}$. Finally, notice that $D(x,y)$ does not contain any isolated vertices, because its vertex set $S(x,y)$ spans all of its edges, by definition. Therefore, $\abs{D(x,y)} \ge \abs{S(x,y)}/2$. At the end of the procedure, the sets $F(x,y), D(x,y), S(x,y)$ are renamed as $F'(x,y), D'(x,y), S'(x,y)$. Note also that if a pair $\{x,y\}$ never becomes $q$-sparse
in the process then $S'(x,y)=D'(x,y)=F'(x,y)=\emptyset$.

\begin{observation}
\label{basicpropertiesofProcedure}
For every $x,y \in A$ and $x,y \in B$, we have
\begin{enumerate}
    \item[(1)] $\abs{S'(x,y)} < q$.
    \item[(2)] $\abs{D'(x,y)} \le \abs{S'(x,y)}$.
    \item[(3)] $\abs{F'(x,y)} = 2 \abs{S'(x,y)}$.
    \item[(4)] $\abs{D'(x,y)} \ge \abs{S'(x,y)}/2$.
\end{enumerate}
\end{observation}

For convenience, throughout the paper we (informally) say that the sets $F'(x,y)$, $D'(x,y)$, $S'(x,y)$ are defined by applying Procedure $\mathcal P(q)$ to a graph $G$ to obtain the graph $G'$, instead of saying that the input to Procedure $\mathcal P(q)$ is $G$ and the output is the graph $G'$ and the sets $F'(x,y)$, $D'(x,y)$, $S'(x,y)$.


\begin{claim}
\label{NowhereObservation}
Let the sets $F'(x,y), D'(x,y), S'(x,y)$ (for $x, y \in A$ and $x, y \in B$) be defined by applying Procedure $\mathcal P(q)$ to a bipartite graph $G$ to obtain $G'$. Let $N(x,y)$ denote the number of common neighbors of vertices $x,y$ in the graph $G$. Then 
$$\frac{\abs{F'(x,y)}}{4} \le \abs{D'(x,y)} < q.$$
Moreover $\abs{F'(x,y)}\leq2\abs{N(x,y)}$.
\end{claim}
\begin{proof}
Combining the parts (3) and (4) of Observation \ref{basicpropertiesofProcedure}, we have $\abs{F'(x,y)}/4 \le \abs{D'(x,y)}$.
Combining the parts (1) and (2) of Observation \ref{basicpropertiesofProcedure}, we obtain $\abs{D'(x,y)} < q$, proving the first part of the claim.

To prove the second part, notice that $S'(x,y)$ is a common neighborhood of $x,y$ in some subgraph $\G$ of $G$, we have $\abs{S'(x,y)}\le \abs{N(x,y)}$. Combining this with part (3) of Observation \ref{basicpropertiesofProcedure}, we obtain $\abs{F'(x,y)}\leq2\abs{N(x,y)}$, as required.
\end{proof}

Finally, let us note the following properties of the graph obtained after applying the procedure.

\begin{observation}
\label{PropertiesOfG'}
Let the sets $F'(x,y), D'(x,y), S'(x,y)$ (for $x, y \in A$ and $x, y \in B$) be defined by applying Procedure $\mathcal P(q)$ to a bipartite graph $G$ to obtain $G'$. Then
\begin{enumerate}
    \item 
    Every edge $ab$ in $G'$ is contained in at most $q/2$ members of  $\{F'(x,y) : x, y \in A\}$ and in at most $q/2$ members of $\{F'(x,y) : x, y \in B\}$. 
    
    \item For any set $M$ of edges in $G'$, removing the edges of $M$ from $G'$ decreases the number of $q$-dense pairs by less than $\abs{M}/2$.
 
\end{enumerate}

\end{observation}

\begin{definition}
Let $H$ be a $3$-partite $3$-graph with  parts $A, B$ and $C$. 

For each $1 \le i \le n$, let $G_i[H](A,B)$ be the bipartite graph with parts $A$ and $B$, whose edge set is $\{ab \mid a \in A, b \in B,  abc_i \in E(H) \}$. The graphs $G_i[H](B, C)$ and $G_i[H](A, C)$ are defined similarly.
\end{definition}

\begin{definition}[Applying Procedure $\mathcal{P}(q)$ to a hypergraph]
\label{ApplyingP(q)toHypergraph}
Let $H$ be a $3$-partite $3$-graph with parts $A, B$ and $C$. We define the hypergraph $H'$ as follows:

For each $1 \le i \le n$, let $G'_i[H](A,B)$, $G'_i[H](B, C)$, $G'_i[H](A,C)$ be the graphs obtained by applying the procedure $\mathcal {P}(q)$ to the graphs $G_i[H](A,B)$, $G_i[H](B, C)$, $G_i[H](A, C)$ respectively. 

For each edge $ab$ which was removed from $G_i[H](A,B)$ by the procedure $\mathcal{P}(q)$ (i.e. $ab\in G_i[H](A,B)\setminus G'_i[H](A,B)$) we remove the hyperedge $abc_i$ from $\h$ (it may have been removed already). 
Similarly for each edge $bc$ (resp. $ac$) which was removed from $G_i[H](B,C)$ (resp. $G_i[H](A,C)$) by the procedure $\mathcal{P}(q)$ we remove the hyperedge $a_ibc$ (resp. $ab_ic$) from $\h$. Let the resulting hypergraph be $H'$. More precisely,
$$H' =\{a_ib_jc_k \in H \mid a_ib_j\in G'_k[H](A,B)\text{, } b_jc_k\in G'_i[H](B,C) \text{, }  a_ic_k\in G'_j[H](A,C)\}.$$

We say $H'$ is obtained from $H$ by applying the Procedure $\mathcal P(q)$.
\end{definition}

\begin{remark}
\label{AB,BC,CAreducedtoAB}
Let $H'$ be obtained by applying the Procedure $\mathcal P(q)$ to the hypergraph $H$.
Then,
\begin{multline*}
\abs{H} - \abs{H'}  \le \sum_{1 \le i \le n} \left ( \abs{G_i[H](A,B)} -\abs{G'_i[H](A,B)} \right ) + \sum_{1 \le i \le n} \left ( \abs{G_i[H](B,C)}-\abs{G'_i[H](B,C)} \right )\\ + \sum_{1 \le i \le n} \left ( \abs{G_i[H](A,C)} - \abs{G'_i[H](A,C)} \right ).
\end{multline*}
Indeed, if $a_ib_jc_k\in H \setminus H'$ then it is easy to see that $a_ib_j\in G_k[H](A,B)\setminus G'_k[H](A,B)$ or $b_jc_k\in G_i[H](B,C)\setminus G'_i[H](B,C)$ or $a_ic_k\in G_j[H](A,C)\setminus G'_j[H](A,C)$.
\end{remark}

\begin{lemma}
    \label{steps}
    Let $q\geq 2$ be an even integer and $G$ be a bipartite graph with parts $A$ and $B$. Suppose $G'$ is the graph obtained by applying Procedure $\mathcal P(q)$ to $G$. Then $G'$ is $K_{2,q}$-free.
\end{lemma}
\begin{proof}
Let us define a \emph{$q$-broom} of size $k$ to be a set of $q$-sparse pairs $\{x_0,x_j\}$ (with $1 \le j \le k$), and a $q$-dense pair $\{y,z\}$ such that $\{y,z\}$ is contained in the common neighborhood of $x_0, x_j$ for every $1 \leq j \leq k$. Note that either $\{x_0, x_1, \ldots, x_k\} \subseteq A$ and $\{y,z\} \subseteq B$ or $\{x_0, x_1, \ldots, x_k\} \subseteq B$ and $\{y,z\} \subseteq A$. 

\begin{claim}
    \label{nobroom}
There is no $q$-broom of size $q/2$ in $G'$. 
\end{claim}
\begin{proof}
Suppose by contradiction that there is a set of $q$-sparse pairs $\{x_0,x_j\}$ (with $1 \le j \le q/2$), and a $q$-dense pair $\{y,z\}$ such that $\{y,z\}$ is contained in the common neighborhood of $x_0$ and $x_j$ for every $1 \leq j \leq q/2$. Then the edge $x_0y$ is contained in the sets $F'(x_0,x_j)$ for every $1 \le j \le q/2$, which contradicts Observation \ref{PropertiesOfG'}.
\end{proof}

Let us suppose for a contradiction (to Lemma \ref{steps}) that $G'$ contains a copy of $K_{2,q}$. Then  $G'$ contains at least one $q$-dense pair. Without loss of generality we may assume there is a $q$-dense pair $\{a,a_1\}$ in $A$.  
Suppose $\{a,a_j\}$ (for $1\leq j \leq p$) are all the $q$-dense pairs of $G'$ containing the vertex $a$. 
For each $1\leq j\leq p$,
let $B_j\subseteq B$ be the common neighborhood of  $a$ and $a_j$ in $G'$. By definition,  $|B_j|\geq q$ for $1\leq j \leq p$. 
\begin{claim}
    \label{hallscondition1}
For any $J\subseteq \{1,2,\ldots, p\}$, we have $\abs{\bigcup_{j\in J} B_j}>2\abs{J}$.
\end{claim}
\begin{proof}
Let us assume for contradiction that there exists a $J\subseteq \{1,2,\ldots, p\}$ such that
$\abs{\bigcup_{j\in J} B_j}\leq 2\abs{J}$. Let $G^*$ be obtained 
from $G'$ by deleting all the edges from $a$ to $\bigcup_{j\in J} B_j$.
For each $j\in J$, the pair $\{a,a_j\}$ has no common neighbor in $G^*$ since we have
removed all the edges from $a$ to $B_j$. Thus the pair $\{a,a_j\}$ is not $q$-dense in $G^*$. So in forming $G^*$ from $G'$ the 
number of $q$-dense pairs decreases by at least $|J|$, while the number of edges
decreases by $|\bigcup_{j\in J} B_j| \le 2|J|$ edges, contradicting Observation \ref{PropertiesOfG'}.
\end{proof}

Let $B'=\bigcup_{1\leq j \leq p} B_j$. For each vertex $v\in B'$ and let 
$$J(v) :=\{j \mid v \in B_j\},$$
$$D(v):= \{ \{v,u\} \mid \{v,u\} \text{ is $q$-dense in $G'$ and }  \{v,u\}\subseteq B_j \text{ for some }
j\in J(v) \}.$$

In the next two claims, we will prove two useful inequalities concerning $\abs{J(v)}$ and $\abs{D(v)}$.

\begin{claim}
    \label{hallscondition2}
For each $v\in B'$, $\abs{J(v)}>2\abs{D(v)}$.
\end{claim}

\begin{proof}
Suppose for contradiction that there is a vertex $v\in B'$ such that $\abs{J(v)}\leq 2\abs{D(v)}$. 
Let us delete all the edges of the form $va_j$, $j \in J(v)$, from $G'$ and let the resulting graph be $G^*$. Since we deleted $\abs{J(v)}$ edges, by Observation \ref{PropertiesOfG'}, the number of  $q$-dense pairs decreases by less than $\abs{J(v)}/2\leq \abs{D(v)}$.
So there exists $\{v,u\}\in D(v)$ such that $\{v,u\}$ is (still) $q$-dense in  $G^*$. That is, $|N^*(v, u)| \geq q$, where $N^*(v,u)$ denotes the common neighborhood of $v$ and $u$ in $G^*$. Clearly each
pair of vertices in $N^*(v,u)$ is contained in a copy of $K_{2,q}$ in $G^*$ (and hence in $G'$).

For each pair of vertices in $N^*(v,u)$, since it is contained in a copy of $K_{2,q}$ in $G'$, it is
either $q$-sparse or $q$-dense in $G'$. Note that $a\in N^*(v,u)$.
If all the pairs $\{a,x\}$ with $x\in N^*(v,u)\setminus \{a\}$ are $q$-sparse in $G'$ then the set of these pairs together with $\{v, u\}$ is a $q$-broom of size at least $q-1 \ge q/2$ in $G'$, which contradicts Claim \ref{nobroom}. So there exists a vertex $x \in N^*(v,u)\setminus \{a\}$ such that $\{a,x\}$ is $q$-dense in $G'$. Since $v$ is adjacent to both $a$ and $x$, 
by the definition of $J(v)$, $x=a_j$ for some $j\in J(v)$. However, by definition, in forming $G^*$ we have removed $vx$ from $G'$. This contradicts $x\in N^*(v,u)$ and completes the proof.
\end{proof}


\begin{claim}
    \label{lowerboundsummv}
$$\sum_{v\in B'} \abs{D(v)} \geq \frac {1}{2} \sum_{1\le j \le p} |B_j|.$$
\end{claim}
\begin{proof} Fix any $j$ with $1\leq j \leq p$.
Since $\{a,a_j\}$ is $q$-dense in $G'$,
every pair $\{x,y\} \subseteq B_j$ is contained in some copy of $K_{2,q}$ and hence 
is  either $q$-dense or $q$-sparse in $G'$.  Let $v$ be any vertex in $B_j$ and let $S(v)=\{y\in B_j \mid \{v,y\} \text{ is  $q$-sparse in $G'$}\}$. By definition, the set $\{\{v,y\} \mid y\in S(v)\}$ together with $\{a,a_j\}$ is a $q$-broom of size $\abs{S(v)}$. By Claim \ref{nobroom}, $|S(v)| \leq q/2-1\leq \abs{B_j}/2-1$. Since $\abs{D(v)}+\abs{S(v)} \ge \abs{B_j}-1$, we have 
\begin{equation} 
    \label{lowerboundmv}
    \abs{D(v)}\geq \frac {1}{2} \abs{B_j}
\end{equation}
Note that \eqref{lowerboundmv} holds for every $j=1,\ldots, p$ and every $v\in B_j$.

Let us define an auxiliary bipartite graph $G_{aux}$ with a bipartition $(\{1,2,\ldots p\},B')$ in which
a vertex $j\in \{1,\ldots, p\}$ is joined to a vertex $y\in B'$ if and only if $y\in B_j$.
Let $J$ be an arbitrary subset of $\{1,2,\ldots, p\}$. The neighborhood of $J$ in $G_{aux}$ is precisely
$\bigcup_{j\in J} B_j$. By Claim \ref{hallscondition1},   $\abs{\bigcup_{j\in J} B_j} > 2\abs{J} \geq \abs{J}$. Since this holds for every $J\subseteq\{1,\ldots, p\}$, by Hall's theorem \cite{Hall} there exist distinct vertices  $w_j \in B_j$, for $j=1,\ldots, p$.  
By \eqref{lowerboundmv},  for every 
$j\in \{1,\ldots, p\}, \abs{D(w_j)}\geq
\frac{1}{2}\abs{B_j}$. Hence
$$\sum_{v\in B'}\abs{D(v)}\geq \sum_{1\leq j \leq p}\abs{D(w_j)}\geq \frac {1}{2} \sum_{1\le j \le p} |B_j|.$$
\end{proof}

If we view $\{B_1,\dots, B_p\}$ as a hypergraph on the vertex set $B'$, then the degree of a vertex $v\in B'$ in it is precisely $\abs{J(v)}$ and the degree sum formula yields
\begin{equation}
\label{eq:sumBj}
\sum_{v\in B'}\abs{J(v)} =\sum_{1\leq j \leq p} \abs{B_j}.
\end{equation}

Using Claim \ref{hallscondition2} and Claim \ref{lowerboundsummv} we have 
$$\sum_{v\in B'}\abs{J(v)}> \sum_{v\in B'}2\abs{D(v)}\geq 2  \sum_{1\le j \le p} \frac {1}{2} |B_j|=\sum_{1\le j \le p} |B_j|,$$ 
which contradicts \eqref{eq:sumBj}. 
This completes proof of Lemma \ref{steps}.
\end{proof}

In the next subsection we will prove a general lemma about making an arbitrary hypergraph $K_{1,2,q}$-free (for any given value of $q$). This lemma is used several times in the following subsections.


\subsection{Applying Procedure $\mathcal P(q)$ to an arbitrary hypergraph $H$}
Let $q$ be an even integer and let $q \geq t$. Let $H$ be an arbitrary $K_{2,q}^{(3)}$-free $3$-partite $3$-graph with parts $A, B$ and $C$. In this subsection we will prove the following lemma that estimates the number of edges removed from the graphs $G_i = G_i[H](A,B)$ for $1 \le i \le n$, when the Procedure $\mathcal P(q)$ is applied to them. This lemma together with Remark \ref{AB,BC,CAreducedtoAB} will allow us to estimate the number of edges removed from $H$ when the Procedure $\mathcal P(q)$ is applied to it.

Throughout this subsection, $N_i(x,y)$ denotes the set of common neighbors of the vertices $x,y$ in the graph $G_i$.

\begin{lemma}
\label{graphupperbound}
Let $q\geq t$ be an even integer. Let $H$ be an arbitrary $K_{2,q}^{(3)}$-free $3$-partite $3$-graph with parts $A, B$ and $C$. Let $G_i=G_i[H](A,B)$ for $1 \le i \le n$. For each $1 \le i \le n$ and any $x,y\in A$ or $x,y\in B$, let $F'_i(x,y)$ be defined by applying the procedure $\mathcal{P} (q)$ to $G_i$ and let the resulting graph be $G_i'$. Then,
$$\sum_{1\leq i \leq n}\abs{G_i\setminus G_i'}<\frac{2}{q} \left(\sum_{u,v\in A}\sum_{1\leq i\leq n}\abs{F'_i(u,v)}+\sum_{u,v\in B}\sum_{1\leq i\leq n}\abs{F'_i(u,v)} \right)+2tn^2.$$
\end{lemma}

\begin{proof}[Proof of Lemma \ref{graphupperbound}]

First let us prove the following claim.
\begin{claim}
    \label{densemultiplicity}
 Let $u,v\in A$ or $u,v\in B$. Then $\{u,v\}$ is $q$-dense in less than $t$ of the graphs  $G_i$, $1 \leq i \leq n$.
\end{claim}
\begin{proof}
Without loss of generality, suppose that $u,v\in A$. Suppose for contradiction that $\{u,v\}$ 
is $q$-dense in $t$ of the graphs $G_i$, $1 \leq i \leq n$. Without loss of generality suppose $\{u,v\}$
is $q$-dense in $G_1,\ldots, G_t$. Then $\abs{N_i(u,v)}\geq q\geq t$ for $i=1,\ldots, t$.
Therefore, we can greedily choose $t$ distinct vertices $y_1,\ldots, y_t$ 
such that for each $i\in [t], y_i\in N_i(u,v)$. For each $i\in [t]$,
since $y_i\in N_i(u,v)$ we have $uy_ic_i,vy_ic_i\in E(H)$. 
However, the set of hyperedges $\{uy_ic_i,vy_ic_i\in E(H) \mid 1\leq i \leq t\}$ forms a copy of $K_{2,t}^{(3)}$ in $H$, a contradiction.
\end{proof}

Note that when procedure $\mathcal {P}(q)$ is applied to $G_i$ (to obtain $G'_i$), Step 1 and Step 2 may be applied several times (and each time one of these steps is applied it may delete an edge of $G_i$). 

For each $i\in [n]$, let $m_i$ denote the number of $q$-dense pairs of $G_i$.
By Claim \ref{densemultiplicity}, we know that each pair $\{u,v\}$ with $u,v\in A$ or $u,v\in B$, is $q$-dense in less than $t$ different graphs $G_i$ (for $1\leq i \leq n$). Therefore, 
\begin{equation}
\label{eq:estimatingm_i}
\sum_{1\leq i \leq n} m_i \le \sum_{u,v\in A}(t-1)+\sum_{u,v\in B}(t-1)=2\binom{n}{2}(t-1).
\end{equation}

For each $i\in [n]$,
let $\alpha_i$ denote the total number of edges that were removed by Step 1 when procedure $\mathcal {P}(q)$ is applied to $G_i$ and $\beta_i$ be the number of edges removed by Step 2 when procedure $\mathcal {P}(q)$ is applied to $G_i$. 
Then $\alpha_i+\beta_i=\abs{G_i\setminus G'_i}$, so $\sum_{i=1}^n \alpha_i + \sum_{i=1}^n \beta_i = \sum_{i=1}^n \abs{G_i\setminus G'_i}$.

First, we bound $\sum_{i=1}^n \beta_i$. Let $i\in [n]$.
Observe that whenever a set $M$ of edges were removed by Step 2 of Procedure $\mathcal P(q)$
applied to $G_i$, the number of $q$-dense pairs decreased by at least $\abs{M}/2$. Hence
$\beta_i \le 2 m_i$. So summing up over all $1 \le i \le n$, and using \eqref{eq:estimatingm_i}, we get
\begin{equation}
    \label{step2edges}
    \sum_{1\leq i \leq n}\beta_i\leq 2\sum_{1\leq i \leq n}m_i\leq 2n(n-1)(t-1)<2tn^2.
\end{equation}

Next, we bound $\sum_{i=1}^n \alpha_i$.
Let $i\in [n]$. If an edge $xy$ was removed from $G_i$ by Step 1 of the procedure $\mathcal {P}(q)$ then there are vertices $z_1,z_2,\ldots ,z_{q/2}$ such that $xy \in F'_i(x,z_j)$ for every $j\in\{1,2,\ldots ,q/2\}$ or $xy \in F'_i(y,z_j)$ for every $j \in \{1,2,\ldots ,q/2\}$. So

$$\alpha_i\leq \frac{1}{q/2} \left (\sum_{u,v\in A}\abs{F'_i(u,v)}+\sum_{u,v\in B}\abs{F'_i(u,v)} \right).$$
Therefore, $$\sum_{1\leq i\leq n}\alpha_i\leq \frac{2}{q} \left (\sum_{1\leq i\leq n}\sum_{u,v\in A}\abs{F'_i(u,v)}+\sum_{1\leq i\leq n}\sum_{u,v\in B}\abs{F'_i(u,v)} \right).$$

This is equivalent to the following. 

\begin{equation}
    \label{sumoffi}
    \sum_{1\leq i\leq n}\alpha_i\leq \frac{2}{q} \left(\sum_{u,v\in A}\sum_{1\leq i\leq n}\abs{F'_i(u,v)}+\sum_{u,v\in B}\sum_{1\leq i\leq n}\abs{F'_i(u,v)} \right).
\end{equation}

Combining this inequality with \eqref{step2edges} completes the proof of Lemma \ref{graphupperbound}.
\end{proof}

\subsection {The overall plan}

 Let us define the sequence $q_0,q_1,\ldots, q_k$ as follows. Let $q_0=2^l$ where $l$ is an integer such that $q_0=2^l\leq t^2<2^{l+1}=2q_0$. For each $1\leq j \leq k$, let $q_{j}=\frac{q_{j-1}}2$ and $q_k\geq t >\frac{q_k}2$. Clearly $\frac{q_0}{q_k}=2^k$, moreover $$2^k=\frac{q_0}{q_k}\leq \frac{t^2}{t}=t.$$ So we have 
 \begin{equation}
    \label{boundingk}
k\leq \log t.
 \end{equation}
 
Now we apply the procedure $\mathcal{P}(q_0)$ to the hypergraph $H$ (recall Definition \ref{ApplyingP(q)toHypergraph})
to obtain a $K_{1,2,q_0}$-free hypergraph $H_0$. For each $0 \leq j < k$ we obtain $K_{1,2,q_{j+1}}$-free hypergraph $H_{j+1}$ by applying the procedure $\mathcal{P}(q_{j+1})$ to the hypergraph $H_{j}$.

This way, in the end we will get a $K_{1,2,q_k}$-free hypergraph $H_k$.
In the following section, we will upper bound $\abs{H}-\abs{H_0}$. Then in the next section, using the information that $H_{j}$ is $K_{1,2,q_{j}}$-free, we will upper bound  $\abs{H_{j+1}}-\abs{H_{j}}$ for each $0\leq j < k$. Then we sum up these bounds to upper bound the total number of deleted edges (i.e., $\abs{H}-\abs{H_k}$) from $H$ to obtain $H_k$. Finally, we bound the size of $H_k$, which will provide us the desired bound on the size of $H$.


\subsection{Making $H$ $K_{1,2,q_0}$-free}
First, we are going to prove an auxiliary lemma that is similar to Lemma A.4 of \cite{MV}.
In an edge-colored multigraph $G$, an \emph{$s$-frame} is a collection of $s$ edges all of different colors
such that it is possible to pick one endpoint from each edge with all the selected endpoints being distinct.

\begin{lemma} \label{frame}
Let $G$ be an edge-colored multigraph with $e$ edges such that each edge has multiplicity at most $p$ and
each color class has size at most $q$. If $G$ contains no $t$-frame then $|G|\leq \binom{t-1}{2}p+tq$.
\end{lemma}
\begin{proof}
Consider a maximum frame $S$, say with edges $e_1,\ldots, e_s$ such that for every $i \in \{1,2, \ldots, s\}$, $e_i$ has color $i$ and
that there exist $x_1\in e_1,x_2\in e_2,\ldots, x_s\in e_s$ with $x_1,\ldots, x_s$ being distinct. 
By our assumption, $s \leq  t-1$. 
Let $f$ be any edge with a color not in $[s]$. Then both vertices of $f$ must be in $\{x_1,\ldots, x_s\}$, otherwise
$e_1,\ldots, e_s, f$ give a larger frame, a contradiction. On the other hand, each edge with both of its vertices in $\{x_1,\dots, x_s\}$ has
multiplicity at most $p$. Hence there are at most $\binom{s}{2}p$ edges with colors not in $\{1,2,\ldots, s\}$. The number of
edges with color in $\{1,2,\ldots, s\}$ is at most $sq$ by our assumption. 
So $|G|\leq \binom{s}{2}p+sq\leq \binom{t-1}{2}p+tq$.
\end{proof}

Let us recall that $H$ is $3$ partite $K_{2,t}^{(3)}$-free hypergraph with $A,B,C$. For convenience we denote $G_i=G_i[H](A,B)$ where $1\leq i \leq n$. For each $1 \le i \le n$ and any $x,y\in A$ or $x,y\in B$, let $F_i'(x,y)$, $D_i'(x,y)$ and $S_i'(x,y)$ be defined by applying the procedure $\mathcal{P}(q_0)$ on $G_i$ and let the obtained graph be $G_i'$.

First, observe that $t^2/2<q_0\leq t^2$ according to our definition.
\begin{claim}
    \label{sparsemultiplicity}
Let $u,v\in A$ or $u,v\in B$. Then $\sum_{1\leq i \leq n}\abs{F_i'(u,v)}\leq 6t^3$.
\end{claim}
\begin{proof}

Let $D^*$ be an edge-colored multigraph in which a pair of vertices $e$ is an edge of color $i\in [n]$
whenever $e$ is an edge of $D'_i(u,v)$. The number of edges of color $i$ in $D^*$ is $\abs{D_i'(u,v)}$.  By Claim \ref{NowhereObservation} we have $\abs{D'_i(u,v)}<q_0$. Hence the number of edges in each color class of $D^*$ is less than $q_0$. 

Let $xy$ be an arbitrary edge of $D^*$ and let $I=\{i \in [n] \mid xy\in D'_i(u,v)\}$ . For each $i\in I$, the pair $\{x,y\}$ is $q_0$-dense in $G_i$ by the definition of $D'_i(u,v)$. Therefore, by Claim \ref{densemultiplicity}, we have $\abs{I}<t$. So $xy$ has multiplicity less than $t$ in $D^*$. Since
$xy$ is arbitrary, the multiplicity of each edge of $D^*$ is less than $t$. 

Next, observe that $D^*$ contains no $t$-frame. Indeed, otherwise without loss of generality we may assume  that $D^*$ contains $t$ edges $x_1y_1,\ldots,x_ty_t$, where $x_iy_i$ has color $i$ for each
$i\in [t]$ and $y_1,\dots, y_t$ are distinct. For each $i\in [t]$ since $x_iy_i\in D'_i(u,v)$, in particular
$y_i\in N_i(u,v)$ (where $N_i(u,v)$ denotes the common neighborhood of $u$ and $v$ in $G_i$), which means that $uy_ic_i, vy_ic_i\in H$. But now, $\{uy_ic_i, vy_ic_i \mid i\in [t]\}$ forms 
a copy of $K^{(3)}_{2,t}$, contradicting $H$ being $K^{(3)}_{2,t}$-free.

Therefore, applying Lemma \ref{frame}, we have $\abs{D^*}\leq \binom{t-1}{2}t+tq_0$.
By Claim \ref{NowhereObservation}, we have $$\frac{\abs{F'_i(u,v)}}{4}\le \abs{D'_i(u,v)}.$$ So
$$\sum_{1\leq i \leq n}\frac{\abs{F'_i(u,v)}}{4}\leq \sum_{1\leq i \leq n}\abs{D'_i(u,v)}=\abs{D^*}\leq \binom{t-1}{2}t + tq_0<\frac32t^3, $$
which proves the claim.
\end{proof}

By Lemma \ref{graphupperbound} we have 
$$\sum_{1\leq i \leq n}\abs{G_i\setminus G_i'}<\frac{2}{q_0} \left(\sum_{u,v\in A}\sum_{1\leq i\leq n}\abs{F'_i(u,v)}+\sum_{u,v\in B}\sum_{1\leq i\leq n}\abs{F_i'(u,v)} \right)+2tn^2.$$

Combining it with Claim \ref{sparsemultiplicity} we get
$$\sum_{1\leq i \leq n}\abs{G_i\setminus G_i'}<\frac{2}{q_0} \left(\sum_{u,v\in A}6t^3+\sum_{u,v\in B}6t^3 \right)+2tn^2.$$


Therefore, as $q_0>t^2/2$, we have
$$\sum_{1\leq i \leq n}\abs{G_i\setminus G_i'}<\frac{4}{t^2} \left( 12t^3\binom{n}{2} \right)+2tn^2<26tn^2.$$
So, $$\sum_{1\leq i \leq n}\abs{G_i\setminus G_i'}=\sum_{1\leq i \leq n}\abs{G_i[H](A,B)\setminus G'_i[H](A,B)}<26tn^2.$$
By symmetry, using the same arguments, we have $$\sum_{1\leq i \leq n}\abs{G_i[H](B,C)\setminus G'_i[H](B,C)}<26tn^2, $$ and
$$\sum_{1\leq i \leq n}\abs{G_i[H](A,C)\setminus G'_i[H](A,C)}<26tn^2.$$
Therefore, by Remark \ref{AB,BC,CAreducedtoAB}, we have
\begin{equation}
    \label{hminush0}
    \abs{H}-\abs{H_0}<78tn^2.
\end{equation}


\subsection{Making a $K_{1,2,q_j}$-free hypergraph $K_{1,2,q_{j+1}}$-free}

In this subsection, we fix a $j$ with $0 \le j < k$. Recall that $H_j$ is $K_{1,2,q_j}$-free, and $H_{j+1}$ is obtained by applying the $\mathcal{P}(q_{j+1})$ to $H_j$. Our goal in this subsection is to estimate $\abs{H_j} - \abs{H_{j+1}}$. The key difference between arguments in this subsection and
in the previous subsection is that now in addition to $H_j$ being $K^{(3)}_{2,t}$-free we
can also utilize the fact that $H_j$ is $K_{1,2,q_j}$-free. In particular, this extra condition leads to 
Claim \ref{improvedlemma}, which improves upon Claim \ref{sparsemultiplicity}.


For convenience of notation, in this subsection, let $G_i=G_i[H_{j}](A,B)$ for each $1\leq i\leq n$. 
For every $1\leq i\leq n$ and every $u,v\in A$ or $u,v\in B$ let the sets $F'_i(u,v)$ and $D'_i(u,v)$ be defined by applying the procedure $\mathcal{P}(q_{j+1})$ to the graph $G_i$, to obtain the graph $G_i'$. 

\begin{claim}
    \label{improvedlemma}
Let $u,v\in A$ or $u,v\in B$. Then $\sum_{1\leq i \leq n}\abs{F_i'(u,v)} < 2q_jt$.
\end{claim}
\begin{proof}
For each $i\in [n]$ we denote the set of common neighbors of $u,v$ in $G_i$ as $N_i(x,y)$. 
For each $i\in [n]$,  since $H_j$ is $K_{1,2,q_j}$-free, $G_i$ is $K_{2,q_j}$-free and so $\abs{N_i(u,v)}<q_j$.

Without loss of generality let us assume $u,v\in A$. For each vertex $w$ of $B$, let $I_w = \{i\in\{1,2,\ldots , n\} \mid w\in N_i(u,v)\}$. We claim that $\abs{I_w} < q_{j}$. Indeed, for each $i\in I_w$, we have $uwc_i,vwc_i\in H_j$. So the set of hyperedges $\{uwc_i,vwc_i \mid i\in I_w\}$ form a copy of $K_{1,2,\abs{I_w}}$ in $H_j$. Thus if $\abs{I_w} \ge q_{j}$, then $H_j$ contains a copy of $K_{1,2,q_{j}}$, a contradiction. Therefore, $\abs{I_w} < q_{j}$, as desired.

Consider an auxiliary bipartite graph $G_{AUX}$ with parts $B$ and $[n]$ where the vertex $i \in [n]$ is adjacent to $b \in B$ in $G_{AUX}$ if and only if $b \in N_i(u,v)$. Then by the discussion in the previous paragraph, each vertex $w \in B$ has degree $\abs{I_w} < q_{j}$, and each vertex $i \in [n]$ has degree $\abs{N_i(u,v)} < q_j$. In other words, the maximum degree in $G_{AUX}$ is less than $q_{j}$. 

We claim that $G_{AUX}$ does not contain a matching of size $t$. Indeed, suppose for a contradiction that the edges $i_1b_{i_1}, i_2b_{i_2}, \ldots, i_tb_{i_t}$ (i.e., $b_{i_l} \in N_{i_l}(u,v)$ for $1 \le l \le t$) form a matching of size $t$ in $G_{AUX}$. Then the set of hyperedges $ub_{i_l}c_{i_l}, vb_{i_l}c_{i_l}$, $1 \le l \le t$, form a copy of $K_{2,t}^{(3)}$ in $H_j$, a contradiction, as desired.

Since $G_{AUX}$ does not contain a matching of size $t$, by the K\"onig-Egerv\'ary theorem
 it has a vertex cover of size less than $t$. This fact combined with the fact that the maximum degree of $G_{AUX}$ is less than $q_j$, implies that the number of edges of  $G_{AUX}$ is less than $q_jt$. On the other hand, the number of edges in $G_{AUX}$ is $\sum_{i\in [n]}\abs{N_i(u,v)}$. Therefore, $\sum_{i\in [n]}\abs{N_i(u,v)} < q_{j}t$. This, combined with the fact that for each $i\in [n]$, $\abs{N_i(u,v)} \geq \abs{F'_i(u,v)}/2$ (see Claim \ref{NowhereObservation}), completes the proof of the lemma.
\end{proof}

By Lemma \ref{graphupperbound}, we have

    $$\sum_{1\leq i\leq n}\abs{G_i\setminus G'_{i}}\leq \frac{2}{q_{j+1}} \left(\sum_{u,v,\in A}\sum_{1\leq i\leq n}\abs{F'_i(u,v)}+\sum_{u,v,\in B}\sum_{1\leq i\leq n}\abs{F'_i(u,v)} \right)+2tn^2.$$
    
    Now using Claim \ref{improvedlemma}, we have
    
$$\sum_{1\leq i\leq n}\abs{G_i\setminus G'_{i}}\leq \frac{8q_{j}t}{q_{j+1}} \binom{n}{2}+2tn^2<\frac {4tq_{j}}{q_{j+1}} n^2+2tn^2.$$
Since $q_{j+1}=q_{j}/2,$ we have
$$\sum_{1\leq i\leq n}\abs{G_i\setminus G'_{i}} < 8t n^2+2tn^2=10tn^2.$$

So, $$\sum_{1\leq i \leq n}\abs{G_i\setminus G'_i}=\sum_{1\leq i \leq n}\abs{G_i[H_j](A,B)\setminus G'_i[H_j](A,B)}<10tn^2.$$
By symmetry, using the same arguments, we have $$\sum_{1\leq i \leq n}\abs{G_i[H_j](B,C)\setminus G'_i[H_j](B,C)}<10tn^2, $$ and
$$\sum_{1\leq i \leq n}\abs{G_i[H_j](A,C)\setminus G'_i[H_j](A,C)}<10tn^2.$$
Therefore, by Remark \ref{AB,BC,CAreducedtoAB}, we have

\begin{equation}
\label{hj1minushj}
  \abs{H_{j}}-\abs{H_{j+1}}<30tn^2.
\end{equation}


\subsection{Putting it all together}
By \eqref{hminush0} and \eqref{hj1minushj} we have 
$$\abs{H}-\abs{H_k}=\abs{H}-\abs{H_0}+\sum_{0\leq j<k}(\abs{H_{j}}-\abs{H_{j+1}})<78tn^2+k(30tn^2).$$

By \eqref{boundingk} we have $k\leq \log{t}$, so we obtain,
\begin{equation}
\label{boundingHminusHk}
\abs{H}-\abs{H_k} < 78tn^2+30t\log{t}n^2.
\end{equation}

Notice that $H_k$ is $K_{1,2,q_k}$-free and $q_k<2t$. Therefore $H_k$ is $K_{1,2,2t}$-free.
Moreover, we know that the hypergraph $H_k$ is $3$-partite and $K_{2,t}^{(3)}$-free with parts $A, B, C$ (as it is a subhypergraph of $H$). Now we bound the size of $H_k$.

\begin{claim}
\label{boundingHk}
We have $\abs{H_k} \le 2tn^2$.
\end{claim}
\begin{proof}
Suppose for a contradiction that $\abs{H_k} > 2t n^2$. For any pair $\{a,b\}$ of vertices with $a \in A$ and $b \in B$, let $\codeg(a,b)$ denote the number of hyperedges of $H_k$ containing the pair $\{a,b\}$. Then the number of copies of $K_{2,1,1}$ in $H_k$ of the form $\{abc, a'bc\}$ where $a,a' \in A$, $b \in B$, $c \in C$ is $$\sum_{\substack{b,c \\ b \in B, c \in C}} \binom{\codeg(b,c)}{{2}}.$$
As the average codegree (over all the pairs $b \in B, c \in C$)  is more than $2t$, by convexity, this expression is more than $$\binom{2t}{2}n^2 > (2t-1)^2 \binom{n}{2}.$$ This means there exist a pair $a, a' \in A$ and a set of $(2t-1)^2+1 > (t-1)(2t-1)+1$ pairs $S := \{bc \mid b \in B, c \in C \}$ such that $abc, a'bc \in E(H_k)$ whenever $bc \in S$. Let $G_{AUX}$ be a bipartite graph whose edges are elements of $S$. Since $G_{AUX}$ has $\abs{S} \ge (t-1)(2t-1)+1$ edges, it either contains a matching $M$ with $t$ edges or a vertex $v$ of degree $2t$ (see Lemma A.3 in \cite{MV}
or the last paragraph of our proof of Claim \ref{improvedlemma} for a proof).
In the former case, the set of all hyperedges of the form $abc, a'bc$ with $bc \in M$, form a copy of $K^{(3)}_{2,t}$ in $H_k$, a contradiction. 
In the latter case, let $u_1, u_2, \ldots, u_{2t}$ be the neighbors of $v$ in $G_{AUX}$. Then the set of hyperedges $\{avu_i, a'vu_i \mid 1 \le i \le 2t \}$ form a copy of $K_{1,2,2t}$ in $H_k$, a contradiction again. This completes the proof of the claim.
\end{proof}

Combining \eqref{boundingHminusHk} with Claim \ref{boundingHk}, we have $\abs{H} \le 80tn^2 + 30t\log{t}n^2,$ thus proving \eqref{eq:reductionto3-partite}, which implies Theorem \ref{mainthm}, as desired.


\section{Concluding remarks}

Recall that given a bipartite graph $G$ with an ordered bipartition $(X,Y)$, where
$Y=\{y_1,\ldots, y_m\}$, $G_{X,Y}^{(r)}$ is the $r$-graph with vertex set $(X\cup Y)\cup (\bigcup_{i=1}^m Y_i)$ and edge set $\bigcup_{i=1}^m\{ e\cup Y_i: e\in E(G), y_i\in e\}$, where $Y_1,\dots,Y_m$
are disjoint $(r-2)$-sets that are disjoint from $X\cup Y$.
A standard reduction argument such as the one used in the proof of Theorem 1.4 in \cite{MV} 
can be used to show the following.
\begin{prop} \label{reduction}
Let $n,r\geq 3$ be integers and $G$ a bipartite graph with an ordered bipartition $(X,Y)$.
There exists a constant $c_r$ depending only on $r$ such that 
\[\ex(n,G^{(r)}_{X,Y})\leq c_r n^{r-3} \cdot \ex(n, G^{(3)}_{X,Y}).\]
\end{prop}
Thus, by Theorem \ref{mainthm} and Proposition \ref{reduction}, for all $r\geq 4$, we have
$\ex(n, K_{2,t}^{(r)})\leq c_r t\log t \binom{n}{r-1}$ for some constant $c_r$, depending only
on $r$. On the other hand, taking the family of all $r$-element subsets of $[n]$
containing a fixed element shows that $\ex(n, K_{2,t}^{(r)})\geq \binom{n-1}{r-1}$.
Recall that in the $r=3$ case, a better lower bound of $\Omega( t\binom{n}{2})$ was shown by Mubayi and Verstra\"ete \cite{MV}. For $r=4$, we are able to improve the lower bound to $\Omega( t\binom{n}{3})$ as follows.

\begin{prop}
We have $$\ex(n, K_{2,t}^{(4)}) \ge (1+o(1))\frac{t-1}{8} n^{3}.$$
\end{prop}
\begin{proof} (Sketch.)
Consider a $K_{2,t}$-free graph $G$ with $(1+o(1))\frac{\sqrt{t-1}}{2} n^{3/2}$ edges where each vertex has degree  $(1+o(1))\sqrt{(t-1)} \sqrt n$. (Such a graph exists by a construction of F\"uredi \cite{Furedi}.) Let us a define a $4$-graph $H=\{abcd\mid ab,cd \in G$ and $ac,ad,bc,bd \notin G\}$. In other words, let the edges of $H$ be the vertex sets of induced $2$-matchings in $G$.
Via standard counting, it is easy to show that
$\abs{H} = (1+o(1))\frac{t-1}{8} n^{3}$. It remains to show $H$ is $K_{2,t}^{(4)}$-free.
\begin{claim}
\label{onlyclaim}
If $axyz,bxyz \in H$, then there is a vertex $c \in \{x,y,z\}$ such that $ac, bc\in G$.
\end{claim}

\begin{proof} By our assumption, $\{a,x,y,z\}$ and $\{b,x,y,z\}$ both induce a $2$-matching in $G$.  
Without loss of generality, suppose $ax,yz\in G$. If $bx\in G$ then we are done. Otherwise, we have
$by,xz\in G$ or $bz,xy\in G$, both contradicting $\{ax,yz\}$ being
an induced matching in $G$.
\end{proof}

Suppose for contradiction that $H$ has a copy of $K_{2,t}^{(4)}$  with 
edge set $\{ ax_iy_iz_i,bx_iy_iz_i \mid 1\leq i \leq t\}$. By Claim \ref{onlyclaim}, for each $1\leq i\leq t$, there exists a vertex $w_i\in\{x_i,y_i,z_i\}$ such that $aw_i,bw_i\in G$. This yields 
 a copy of $K_{2,t}$ in $G$, a contradiction.
\end{proof}

For $r\geq 5$, we do not yet have a lower bound that is asymptotically larger than $\binom{n-1}{r-1}$. It would be interesting to narrow the gap between the lower and upper bounds on
$\ex(n,K_{2,t}^{(r)})$.

\vspace{2mm}

It will be interesting to have a systematic study of the function $\ex(n,G_{X,Y}^{(r)})$.
Mubayi and Verstra\"ete \cite{MV} showed that $\ex(n,K_{s,t}^{(3)})=O(n^{3-1/s})$ and that
if $t>(s-1)!>0$ then $\ex(n, K_{s,t}^{(3)})=\Omega(n^{3-2/s})$ and speculated that
$n^{3-2/s}$ is the correct order of magnitude.
The case when $G$ is a tree is studied in \cite{FJKMV}, where the problem considered there
is slightly  more general.  The case when $G$ is an even cycle has also been studied.
Let $C_{2t}^{(r)}$ denote $G_{X,Y}^{(r)}$ where $G$ is the even cycle $C_{2t}$
of length $2t$. It was shown by Jiang and Liu \cite{JL} that 
$c_1 t \binom{n}{r-1}\leq \ex(n,C_{2m}^{(r)})\leq c_2 t^5 \binom{n}{r-1}$, for some positive constants
depending $c_1,c_2$ on $r$. Using results in this paper and new ideas,
we are able to narrow the gap to
$c_1 t\binom{n}{r-1} \leq \ex(n, C_{2m}^{(r)})\leq c_2 t^2\log t \binom{n}{r-1}$, for some
positive constants $c_1,c_2$ depending on $r$. We would like to postpone this and other results on the topic  for a future paper.

Finally, motivated by results on $K_{2,t}^{(r)}$ and $C_{2t}^{(r)}$, we pose the following question.

\begin{question}
Let $r\geq 3$. Let $\mathcal G$ be the family of bipartite graphs $G$ with an ordered bipartition
$(X,Y)$ in which every vertex in $Y$ has degree at most $2$ in $G$. Is it true that 
$\forall G\in \mathcal G$ there is a constant $c$ depending  on $G$ such that
$\ex(n,G_{X,Y}^{(r)})\leq c\binom{n}{r-1}$?
\end{question}


\section*{Acknowledgments}

The research of the first and third authors was supported by the Doctoral Research Support Grant of CEU, and by the Hungarian National Research, Development and Innovation Office NKFIH, grant K116769. The first and third authors are especially grateful for the generous hospitality of Miami University.

\end{document}